\documentclass[submission%
]{dmtcs}

\usepackage[latin1]{inputenc}
\usepackage{subfigure}

\usepackage[round]{natbib}

\usepackage[usenames]{color}

\author{
Chris J. Kuhlman\addressmark{1,3}
\and
Henning S. Mortveit\addressmark{1,2}
\thanks{Email: \email{Henning.Mortveit@vt.edu} (corresponding author)}
\and
David Murrugarra\addressmark{2}
\and
V.~S.~Anil Kumar\addressmark{1,3}
}
\title
[Bifurcations in Boolean Networks]
{Bifurcations in Boolean Networks}
\address{
\addressmark{1}Network Dynamics and Simulation Science Laboratory, Virginia Tech\\
\addressmark{2}Department of Mathematics, Virginia Tech\\
\addressmark{3}Department of Computer Science, Virginia Tech
}
\received{1998-10-14}
\revised{\today}
\accepted{tomorrow}

\usepackage{amsmath,amsfonts,amssymb}
\usepackage[all]{xy}
\usepackage{xspace}
\usepackage{color}
\graphicspath{{./}{./figs/}}
\newtheorem{theorem}{Theorem}[section]
\newtheorem{corollary}[theorem]{Corollary}
\newtheorem{lemma}[theorem]{Lemma}
\newtheorem{proposition}[theorem]{Proposition}

\newtheorem{example}[theorem]{Example}
\numberwithin{equation}{section}
\def\F{\mathbf{F}}

\def\N{\mathbb{N}}
\def\R{\mathbb{R}}

\def\<{\langle}
\def\>{\rangle}

\def\bigcard#1{\bigl|#1\bigr|}
\def\Circle{\mathrm{Circ}}

\def\vset{\mathrm{v}}
\def\eset{\mathrm{e}}

\def\kup{k^\uparrow}
\def\kdown{k^\downarrow}
\def\tfunc{\text{\bfseries \emph{t}}}

\def\lcm{\operatorname{lcm}}
\let\emptyset\varnothing

\definecolor{red}{rgb}{1.0,0.0,0.0}

\begin{document}
\maketitle

\begin{abstract}
  This paper characterizes the attractor structure of synchronous and
  asynchronous Boolean networks induced by bi-threshold functions.
  Bi-threshold functions are generalizations of standard threshold
  functions and have separate threshold values for the transitions
  $0\to1$ (up-threshold) and $1\to0$ (down-threshold). We show that
  synchronous bi-threshold systems may, just like standard threshold
  systems, only have fixed points and $2$-cycles as attractors.
  Asynchronous bi-threshold systems (fixed permutation update
  sequence), on the other hand, undergo a bifurcation. When the
  difference $\Delta$ of the down- and up-threshold is less than~$2$
  they only have fixed points as limit sets. However, for $\Delta \ge
  2$ they may have long periodic orbits. The limiting case of $\Delta
  = 2$ is identified using a potential function argument. Finally, we
  present a series of results on the dynamics of bi-threshold systems
  for families of graphs.
\end{abstract}




\section{Introduction}
A standard Boolean \emph{threshold function} $\tfunc_{k,m} \colon
\{0,1\}^m \longrightarrow \{0,1\}$ is defined by
\begin{equation}
\label{eq:threshold}
\tfunc_{k,m}(x_1, \ldots, x_m) =
\begin{cases}
1, & \text{if } \sigma(x_1, \ldots, x_m) \ge k  \text{\quad and} \\
0, & \text{otherwise,}
\end{cases}
\end{equation}
where $\sigma(x_1, \ldots, x_m) = \bigcard{ \{1 \le i \le m \mid x_i =
  1\} }$.
This class of functions is a common choice in modeling biological
systems~[\cite{Kauffman:69,Karaoz:04}], and social behaviors (e.g., joining
a strike or revolt, adopting a new technology or contraceptives,
spread of rumors and stress, and collective action), see,
e.g.,~[\cite{granovetter:granovetter-1978, bdkw-1989, macy-1991,
  centola:cm-2007, watts:watts-2002, kkt-2003}].
A \emph{bi-threshold function} is a function $\tfunc_{i,\kup,\kdown,m}
\colon \{0,1\}^{m} \longrightarrow \{0,1\}$ defined by
\begin{equation}
\tfunc_{i,\kup,\kdown,m}(x_1,\ldots, x_m) =
\begin{cases}
\tfunc_{\kup,m},& \text{if } x_i = 0,\\
\tfunc_{\kdown,m},& \text{if } x_i = 1 \;.
\end{cases}
\end{equation}
Here $i$ denotes a designated argument -- later it will be the vertex
or cell index.  We call $\kup$ the \emph{up-threshold} and $\kdown$
the \emph{down-threshold}. When $\kup=\kdown$ the bi-threshold
function coincides with a standard threshold function. Note that
unlike the standard threshold function in~\eqref{eq:threshold} which
is symmetric, the bi-threshold function is \emph{quasi-symmetric} (or
outer-symmetric) -- with the exception of index~$i$, it only depends
on its arguments through their sum.

In this paper we consider synchronous and asynchronous \emph{graph
  dynamical systems} (GDSs), see~[\cite{Mortveit:07,Macauley:09a}], of
the form $\F \colon \{0,1\}^n \longrightarrow \{0,1\}^n$ induced by
bi-threshold functions. These are natural extensions of threshold GDSs
and capture threshold phenomena exhibiting hysteresis properties.
Bi-threshold systems are also prevalent in social systems where each
individual can change back-and-forth between two states; Schelling
states: ``Numerous social phenomena display cyclic behavior ...'',
see~\cite[p.~86]{schelling-1978}.  Among his examples is whether
pick-up volleyball games will continue through an academic semester or
die (e.g., individuals regularly choosing to play or not play).  One
can also look at public health concerns such as obesity, where an
individual's back-and-forth decisions to diet or not---which are peer
influenced, [\cite{cf-2007}], and therefore can be at least partially
described by thresholds---are so commonplace that it has a name:
``yo-yo dieting''~[\cite{ntf-1994}].  When $\kup > \kdown$, a vertex
that transitions from state $0$ to state $1$ is more likely to remain
in state $1$ than what would be the case in a standard threshold
GDS. For the state transitions from~$1$ to~$0$ the situation is
analogous.  This suggests that the cost to change back to state~$0$ is
great or that a change to state~$0$ will occur only if the conditions
that gave rise to the $0 \rightarrow 1$ transition significantly
diminish.  A company that acquires and later divests itself of a
competitor is such an example.  Examples where $\kdown \ge \kup$ are
commonplace.  For example, [\cite{schelling-1978}] states that he often
witnesses people who start to cross the street against traffic lights,
but will return to the curb if they observe an insufficient number of
others following behind.  Overshooting, whereby a group of individuals
take some action, and within a short time period, a subset of these
pull back from it, is also of interest to the sociology
community~[\cite{bm:bm-2009}] and is characterized by~$\kdown \geq
\kup$.


It is convenient to introduce the quantity $\Delta = \kdown -
\kup$. The first of our main results (Theorem~\ref{thm:bithres-gca})
characterizes limit cycle structure of synchronous bi-threshold GDS
(also known as as Boolean networks). Building on the proof for
threshold functions in~\cite{Goles:81}, we prove that only fixed
points and periodic orbits of length~2 can occur for each possible
combination of $\kup$ and $\kdown$. Since we re-use parts of their
proof, and also since their proof only appears in French, a condensed
English translation is included in the appendix on
page~\pageref{sec:standard}.
The situation is very different for asynchronous bi-threshold GDSs
where a vertex permutation is used for the update
sequence. Our second main result states that when $\Delta < 2$, only
fixed points can occur as limit cycles. However, for $\Delta \ge 2$
there are graphs for which arbitrary length periodic orbits can be
generated. The case $\Delta = 2$ is identified using a potential
function argument and represents a (2-parameter) \emph{bifurcation} in
a discrete system, a phenomenon that to our knowledge is novel. We
also include a series of results for bi-threshold dynamics on special
graph classes. These offer examples of asynchronous bi-threshold GDSs
with long periodic orbits, and may also serve as building blocks in
construction and  modeling of bi-threshold systems with given cycle
structures.

{\bfseries Paper organization.} We introduce necessary definitions
and terminology for graph dynamical systems in
Section~\ref{sec:background}. The two main theorems are presented in
Sections~\ref{sec:thmsynch} and~\ref{sec:thmasynch}. Our collection of
results on dynamics for graph classes like trees and cycle graphs
follow in Section~\ref{sec:dynamics} before we conclude in
Section~\ref{sec:conclude}.


\section{Background  and Terminology}
\label{sec:background}

In the following we let $X$ denote an undirected graph with vertex set
$\vset[X] = \{1,2,\ldots,n\}$ and edge set $\eset[X]$. To
each vertex $v$ we assign a state $x_v \in K = \{0,1\}$ and refer to
this as the \emph{vertex state}. Next, we let $n[v]$ denote the
sequence of vertices in the $1$-neighborhood of $v$ sorted in
increasing order and write
\begin{equation*}
x[v] = (x_{n[v](1)}, x_{n[v](2)}, \ldots, x_{n[v](d(v)+1)} )
\end{equation*}
for the corresponding sequence of vertex states. Here~$d(v)$ denotes
the degree of~$v$. We call $x = (x_1, x_2, \ldots, x_n)$ the
\emph{system state} and $x[v]$ the \emph{restricted state}.
The dynamics of vertex states is governed by a list of \emph{vertex
  functions} $(f_v)_v$ where each $f_v \colon K^{d(v)+1}
\longrightarrow K$ maps as
\begin{equation*}
x_v( t + 1 ) = f_v\bigl( x(t)[v] \bigr)\;.
\end{equation*}
In other words, the state of vertex $v$ at time~$t+1$ is given
by~$f_v$ evaluated at the restricted state $x[v]$ at time~$t$. An
\emph{update mechanism} governs how the list of vertex functions
assemble to a \emph{graph dynamical system}
map (see e.g.~\cite{Mortveit:07,Macauley:09a})
\begin{equation*}
\F \colon K^n \longrightarrow K^n
\end{equation*}
sending the system state at time~$t$ to that at time~$t+1$.

For the update mechanism we will here use \emph{synchronous} and
\emph{asynchronous} schemes. In the former case we obtain Boolean
networks where
\begin{equation*}
\F(x_1, \ldots, x_n ) = (f_1(x[1]), \ldots, f_n(x[n]) ) \;.
\end{equation*}
This sub-class of graph dynamical systems is sometimes referred to as
\emph{generalized cellular automata}. In the latter case we will
consider permutation update sequences. For this we first introduce the
notion of \emph{$X$-local functions}. Here the $X$-local function $F_v
\colon K^n \longrightarrow K^n$ is given by
\begin{equation*}
  F_v(x_1, \ldots, x_n ) = (x_1, x_2, \ldots, f_v(x[v]), \ldots, x_n) \;.
\end{equation*}
Using $\pi = (\pi_1, \ldots, \pi_n) \in S_X$ (the set of all
permutations of $\vset[X]$) as an update sequence, the
corresponding asynchronous (or sequential) graph dynamical system
map $\F_\pi \colon K^n \longrightarrow K^n$ is given by
\begin{equation}
  \label{eq:asynch}
  \F_\pi = F_{\pi(n)} \circ F_{\pi(n-1)} \circ \cdots \circ F_{\pi(1)}\;.
\end{equation}
We also refer to this class of asynchronous systems as (permutation)
\emph{sequential dynamical systems} (SDSs).  The $X$-local functions are
convenient when working with the asynchronous case.
In this paper we will consider graph dynamical systems induced by
bi-threshold functions, that is, systems where each vertex function is
given as
\begin{equation*}
 f_v = f_{v,\kup_v,\kdown_v} := \tfunc_{v,\kup_v,\kdown_v,d(v)+1} \;.
\end{equation*}

The phase space of the GDS map $\F \colon K^n \longrightarrow K^n$ is
the directed graph with vertex set $K^n$ and edge set $\{ \bigl(x, \F(x)\bigr)
\mid x\in K^n\}$. A state $x$ for which there exists a positive
integer $p$ such that $\F^p(x) = x$ is a \emph{periodic point}, and
the smallest such integer $p$ is the \emph{period} of $x$. If $p=1$ we
call $x$ a \emph{fixed point} for $\F$. A state that is not periodic
is a \emph{transient state}. Classically, the \emph{omega-limit set}
of $x$, denoted by $\omega(x)$, is the accumulation points of the
sequence $\{\F^k(x)\}_{k\ge 0}$. In the finite case, the omega-limit
set is the unique periodic orbit reached from~$x$ under~$\F$.
\begin{example}
\label{ex:basic}
To illustrate the above concepts, take $X= \Circle_4$ as graph (shown
in Figure~\ref{fig:ex1}), and choose thresholds $\kup = 1$ and
$\kdown = 3$. For the synchronous case we have we have for example
$\F(1,0,0,1) = (0,1,1,0)$. Using the update sequence $\pi = (1,2,3,4)$
we obtain $\F_\pi(1,0,0,1) = (0,0,1,0)$. The phase spaces of $\F_\pi$ and
$\F$ are shown in Figure~\ref{fig:ex1}. Notice that $\F_\pi$ has
cycles of length~$3$, while the maximal cycle length of $\F$ is~$2$.
\begin{figure}[ht]
\centerline{\includegraphics[width=0.95\textwidth]{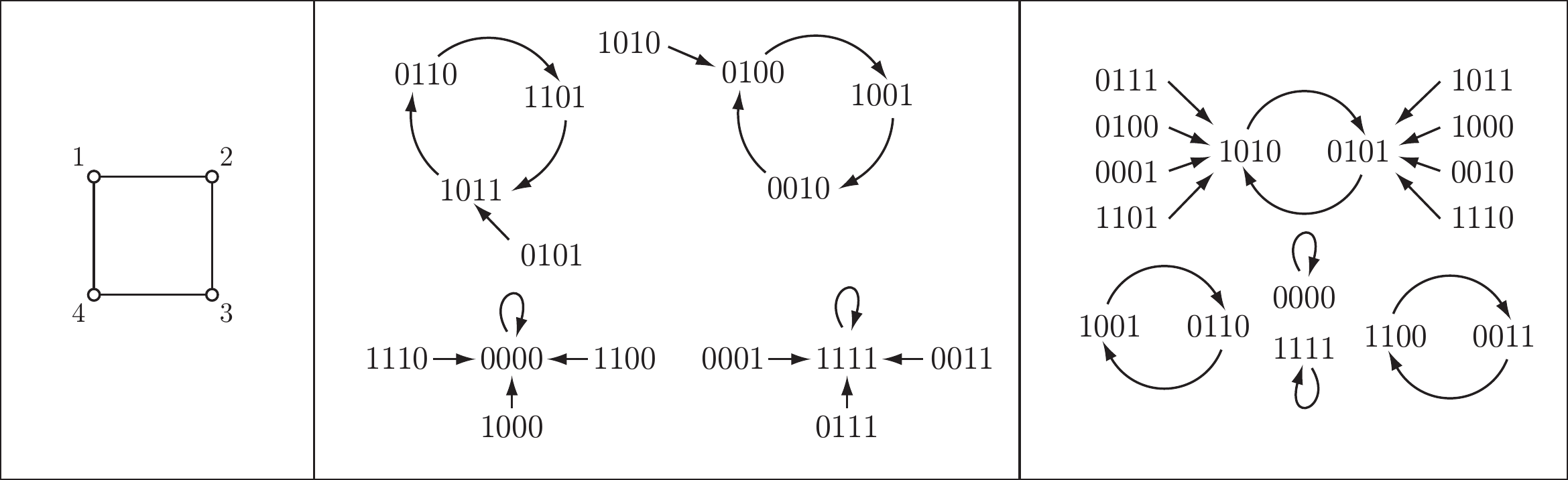}}
\caption{The graph $X = \Circle_4$ (left), and the phase spaces of
  $\F_\pi$ (middle) and $\F$ (right) for Example~\ref{ex:basic}.}
\label{fig:ex1}
\end{figure}
\end{example}
We remark that graph dynamical systems generalize concepts such as
cellular automata and Boolean networks, and can describe a wide range
of distributed, nonlinear phenomena.


\section{$\omega$-Limit Set Structure of Bi-Threshold GDS}

This section contains the two main results on dynamics of synchronous
and asynchronous bi-threshold GDSs.

\subsection{Synchronous Bi-Threshold GDSs}
\label{sec:thmsynch}

Let $K=\{0,1\}$ as before, let $A=(a_{ij})$ be a real-valued symmetric
matrix, let $(\kup_i)_{i=1}^n$ and $(\kdown_i)_{i=1}^n$ be
vertex-indexed sequences of up- and down-thresholds, and
define the function $\F = (f_1,\dots,f_n) \colon K^n
\longrightarrow K^n$ by
\begin{equation}
\label{eq:genbithreshold}
f_i(x_1,\dots,x_n) =
\begin{cases}
   1   & \text{ if } x_i=0 \text{ and }\sum\limits^n_{j=1}a_{ij}x_j \ge \kup_i  \\
   0   & \text{ if } x_i=1 \text{ and }\sum\limits^n_{j=1}a_{ij}x_j < \kdown_i  \\
  x_i  &   \text{ otherwise.}
\end{cases}
\end{equation}
The following theorem is a generalization of
Theorem~\ref{thm:thres-gca} (see appendix) to the case of bi-threshold
functions.
\begin{theorem}
\label{thm:bithres-gca}
If $\F$ is the synchronous GDS map over the complete graph of
order~$n$ with vertex functions as in Equation~\eqref{eq:genbithreshold}, then
for all $x\in K^n$, there exists $s\in\mathbb{N}$ such that
$\F^{s+2}(x) = \F^{s}(x)$.
\end{theorem}

\def\supp{\ensuremath{\operatorname{supp}}}

The proof builds on the arguments of the proof from~\cite{Goles:81}
for standard threshold functions (see page~\pageref{sec:standard} of
the appendix). Note that we can use Lemma~\ref{lem:tprop1} in its
original form, but for Lemma~\ref{lem:tprop2} changes are needed
to adapt for bi-threshold functions. The position is marked {\bfseries
  [Cross-reference for bi-threshold systems]} in the the proof of
Lemma~\ref{lem:tprop2} on page~\pageref{lem:tprop2}. Before starting
the proof of the theorem above, we first introduce the notion of
\emph{bands} and give a result on their structural properties. This is
essential in the extension of the original result.

\medskip

Let $z_i \in S$ and assume that $\gamma_i \ge 3$. As in the proof
of~\cite{Goles:81}, we set
\begin{equation*}
\supp(z_i) = \bigl\{l\in \{0,1,2,\ldots,T-1 \} : z_l = 1 \bigr\};\,
\end{equation*}
and use their partition $\mathcal{C} = \{C_0, C_1, C_2, \ldots,
C_p\}$. By the assumption~$\gamma_i \ge 3$, we are guaranteed that
$p\ge 1$.
The bi-threshold functions require a more careful structural analysis
of the elements of~$\mathcal{C}$ than in the case of standard
threshold functions. We say that $C\in \mathcal{C}$ is of \emph{type}
$ab$ if $z_j = a$ and $z_{j'} = b$ where $j$ and $j'$ are the indices
immediately to the left and right of $C$, respectively (viewed modulo
$T$). Here we write $m_{ab} = m_{ab}(\mathcal{C})$ for the number of
elements of $\mathcal{C}$ of type~$ab$.

We claim that $m_{01} = m_{10}$. Before we prove this, observe first
that the sequence $\bigl(z_i(0), z_i(1), \ldots, z_i(T-1)\bigr)$ can
be split into contiguous sub-sequences (\emph{bands}) whose states
contain only isolated~0s, where the end points have state~1, and where
bands are separated by sub-sequences of lengths~$\ge 2$ whose state
consist entirely of~0s. By the construction of~$\mathcal{C}$, each
element $C\in\mathcal{C}$ must be fully contained in a single
band. Our claim above is now a direct consequence of the following
lemma:

\begin{lemma}
\label{lem:bands}
A band either ($i$) contains no element $C$ of type $01$ or $10$, or
($ii$) contains precisely one element $C$ of type $01$ and precisely
one element $C'$ of type $10$.
\end{lemma}
\begin{proof}
Fix a band $B$ and let $C\in\mathcal{C}$ be the partition containing
the initial element of B. There are now two possibilities.
In the first case, $C$ also contains the final element of $B$. Then
$C$ has type $00$, and any other partition element contained in $B$ is
necessarily of type $11$.
In the second case, $C$ terminates before the end of $B$. The
configuration at the end of $C$ must then be as\hfill\par
\centerline{\includegraphics[]{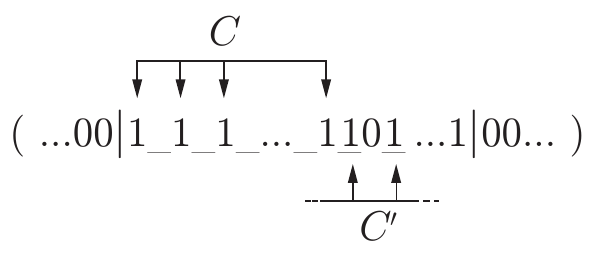}}\hfill\par\noindent and $C$ is
of type~$01$. The element $C'$ containing the index after the last
element of $C$ either goes all the way to the end of $B$, in which case it
is of type $10$, or it terminates before that in which case the
situation is as in the diagram above and $C'$ is of type $11$. By
repeated application of this argument, the band $B$ is eventually
exhausted with an element $C''$ of type $10$. All other elements of
$\mathcal{C}$ within $B$ not included in the sequence of partitions
$C$, $C'$ and so on, must be of type~$11$, and the proof is complete.
\end{proof}

\begin{corollary}
\label{cor:m01-m10}
$m_{01}(\mathcal{C}) = m_{10}(\mathcal{C})$
\end{corollary}

\begin{proof}[(Theorem~\ref{thm:bithres-gca})]
{\bfseries Claim:} If $\gamma_i \ge 3$ for $z_i \in S$ then
$\sum_{j=1}^n L(z_i, z_j) < 0$.\hfill\par
We can write
\begin{equation*}
\sum_{i=1}^n L(z_i, z_j)
= \sum_{k=0}^p \Bigl( \sum_{j=1}^n a_{ij}
\sum_{l\in C_k} \bigl( z_j(l+1) - z_j(l-1)   \bigr) \Bigr)
= \sum_{k=0}^p \Psi_{ik}\;,
\end{equation*}
where
\begin{equation*}
\Psi_{ik} = \sum_{j=1}^n a_{ij} \sum_{l\in C_k} \bigl( z_j(l+1) -
z_j(l-1) \bigr)
= \sum_{j=1}^n a_{ij} z_j(l_k+2q_k+1) - \sum_{j=1}^n a_{ij} z_j(l_k-1) \;.
\end{equation*}
We need to consider $\Psi_{ik}$ for the four types of partition
elements. As in the original proof, note that~$\Psi_{i0} = 0$.

\medskip
\noindent \underline{{\bfseries $C_k$ is of type $00$:}} in this case
$z_i(l_k-1) = 0$, $z_i(l_k) = 1$, $z_i(l_k+2q_k+1) = 0$ and
$z_i(l_k+2q_k+2) = 0$, which is only possible if
\begin{equation*}
\sum_{j=1}^n a_{ij} z_j(l_k-1) \ge \kup_i
\quad \text{and} \quad
\sum_{j=1}^n a_{ij} z_j(l_k+2q_k+1) < \kup_i \;,
\end{equation*}
which implies that $\Psi_{ik} < 0$.

\medskip
\noindent \underline{{\bfseries $C_k$ is of type $11$:}} this case is
completely analogous to the $00$ case, and again we conclude that
$\Psi_{ik} < 0$.

\medskip
\noindent \underline{{\bfseries $C_k$ is of type $10$:}} here
$z_i(l_k-1) = 1$, $z_i(l_k) = 1$, $z_i(l_k+2q_k+1) = 0$ and
$z_i(l_k+2q_k+2) = 0$. This implies that
\begin{equation*}
\sum_{j=1} a_{ij}z_j(l_k-1) \ge \kdown_i
\quad\text{and}\quad
\sum_{j=1} a_{ij}z_j(l_k+2q_k+1) < \kup_i \;,
\end{equation*}
leading to $\Psi_{ik} < \kup_i - \kdown_i$.

\medskip
\noindent \underline{{\bfseries $C_k$ is of type $01$:}} this case is
essentially the same as the $10$ case, but here $\Psi_{ik} < \kdown_i -
\kup_i$.

\medskip

Using the above four cases, we now have
\begin{equation*}
\sum_{j=0}^n L(z_i, z_j) = \sum_{k=0}^p \Psi_{ik}
< 0
+ m_{00} \cdot 0
+ m_{11} \cdot 0
+ m_{10} (\kup_i - \kdown_i)
+ m_{01} (\kdown_i - \kup_i) = 0\;,
\end{equation*}
where the last equality follows by Corollary~\ref{cor:m01-m10}. Clearly,
this leads to the same contradiction as in the proof of
Theorem~\ref{thm:thres-gca}.
\end{proof}

An immediate consequence of Theorem~\ref{thm:bithres-gca} is the following:
\begin{corollary}
A synchronous bi-threshold GDS may only have fixed points and 2-cycles as
limit sets.
\end{corollary}


\subsection{Asynchronous Bi-Threshold GDSs}
\label{sec:thmasynch}

\begin{theorem}
\label{thm:asynch}
Let $X$ be a graph, let $\pi\in S_x$ and let $(f_v)_v$ be bi-threshold
functions all satisfying $\Delta(v) = \kdown_v - \kup_v \le 1$. The
sequential dynamical system map $\F_\pi$ only has fixed points as
limit sets.
\end{theorem}
As before, the graph $X$ is finite. Note also that the per-vertex
thresholds $\kup$ and $\kdown$ need not be uniform for the graph.

\begin{proof}
The proof uses a potential function based on a construction
in~\cite{Barrett:06a}, but see also~\cite{Goles:85}. For a given
state~$x\in K^n$ we assign to each vertex the potential
\begin{equation*}
P(v,x) =
\begin{cases}
\kdown_v, & x_v = 1\\
d(v)+2-\kup_v, & x_v = 0\;.
\end{cases}
\end{equation*}
Note that the quantity $d(v)+2-\kup_v$ is the smallest number of vertex states
in the local state $x[v]$ that must be zero to ensure that $x_v$ remains in
state zero. Similarly, an edge $e = \{v,v'\}$ is assigned the potential
\begin{equation*}
P(e=\{v,v'\},x) =
\begin{cases}
1,& x_v \ne x_{v'} \\
0,& x_v = x_{v'} \;.
\end{cases}
\end{equation*}
For book-keeping, we let $n_i = n_i(v; x)$ denote the number of
vertices adjacent to $v$ in state $i$ for $i = 0,1$ and note that $n_0
+ n_1 = d$. The \emph{system potential} $P(x)$ at the state~$x$ is the
sum of all the vertex and all the edge potentials. For the theorem
statement it is clearly sufficient to show that each application of a
vertex function that leads to a change in a vertex state causes the
system potential to drop.

Consider first the case where $x_v$ is mapped from $0$ to $1$ which
implies that $n_1 \ge \kup_v$. Since a change in system potential only
occurs for vertex $v$ and edges incident with $v$, we may disregard
the other potentials when determining this change. Denoting the system
potential before and after the update by $P$ and $P'$, we have $P =
d+2-\kup_v + n_1$ and $P' = \kdown_v + n_0$ which implies that
\begin{align*}
P' - P &= \kdown_v + n_0 - d - 2 + \kup_v - n_1 = \kdown_v + \kup_v - 2n_1 - 2\\
       &\le -(\kup_v - \kdown_v) -2 = \Delta(v) - 2\;,
\end{align*}
and this is strictly negative whenever $\Delta = \kdown - \kup \le
1$. Similarly, for the transition where $x_v$ maps from $1$ to $0$ one
must have $n_1 + 1 \le \kdown_v - 1$ or $n_1 \le \kdown_v-2$. In this case
we have
\begin{align*}
P' - P &= [d+2-\kup_v+n_1] - [\kdown_v+n_0] = 2n_1 + 2 -\kdown_v - \kup_v \\
       &\le 2\kdown_v - 4 + 2 - \kup_v - \kdown_v = \Delta(v) - 2
\end{align*}
as before, concluding the proof.
\end{proof}


\subsection{Bifurcations in Asynchronous GDS}

A natural question now is what happens in the case where $\Delta =
\kdown - \kup = 2$ since periodic orbits are no longer excluded by the
arguments in the proof above. The following proposition shows that
there are graphs and choices of $\kup$ and $\kdown$, such that~$\Delta
= 2$, for which there are periodic orbits of arbitrary length.
\begin{proposition}
\label{prop:circle}
The bi-threshold GDS map over $X = \Circle_n$ with update sequence
$\pi = (1,2,3,\ldots,n)$, thresholds $\kup = 1$ and $\kdown = 3$, has
cycles of length $n-1$.
\end{proposition}
\begin{proof}
We claim that the state $x = (0,0,\ldots,0,1,0)$ is on an
$(n-1)$-cycle. Straightforward computations give that the single
$1$-state is shifted one position to the left upon each application of
$\F_\pi$ until the state $y = (0,1,0,\ldots,0)$ is reached. The image
of this state is $z = (1,0,0,\ldots, 0,0,1)$ which is easily seen to
map to $x$. The smallest number of iterations required to return to
the original state $x$ is $n-1$, producing a cycle as claimed.
\end{proof}

In other words, by taking $\Delta$ as a parameter, we see that the
bi-threshold sequential dynamical system undergoes a bifurcation at
$\Delta = 2$.


\section{Dynamics of Bi-Threshold GDSs}
\label{sec:dynamics}

\subsection{Graph Unions}

From Proposition~\ref{prop:circle}, we see that for $X = \Circle_n$
with threshold $\kup = 1$ and $\kdown = 3$ at each vertex, we obtain
an $(n-1)$-cycle for the update sequence $\pi = (1,2, \dots, n)$. The
following proposition demonstrates how we can combine graphs to obtain
larger cycle sizes for bi-threshold SDSs with arbitrarily nonuniform
$\kup,\kdown$.  In particular, the result
applies to the case where we combine $\Circle_n$ graphs where~$p =
n-1$ is prime.

\begin{proposition}
\label{prop:union}
For $i=1,2$ let $X_i$ be a graph for which the bi-threshold GDS with
update sequence $\pi_i$ has a cycle in phase space of length $c_i$.
Let $u_i\in \vset[X_i]$, and let $X$ be the graph obtained as the
disjoint union of $X_1$ and $X_2$ plus additionally the vertex
$w\not\in \vset[X_1],\vset[X_2]$ with the edges $\{u_1, w\}$ and
$\{u_2,w\}$. Moreover, let all thresholds of vertices in $X_1$ and
$X_2$ be as before,
and assign threshold $\kup = 3$ to $w$. The
bi-threshold SDS map over $X$ with update sequence $\pi =
(\pi_1|\pi_2|w)$ [juxtaposition] has a cycle of length $\lcm(c_1,
c_2)$.
\end{proposition}

\begin{proof}
Let vertex $w$ have $\kup = 3$, so that $w$ will never transition to
state~$1$ from state~$0$. Let $x = (x_1|x_2|x_w)$ be the state over
$X$ constructed from states $x_1$ and $x_2$ on the respective
$c_i$-cycle over $X_1$ and $X_2$ with $x_w = 0$. The only
vertices whose connectivity, and therefore induced vertex function,
are affected by the addition of $w$ are $u_1$ and $u_2$.  But the
state transitions for $u_1$ and $u_2$ are unaffected because each is
predicated on $\sigma(x[u_1])$ and $\sigma(x[u_2])$, respectively, and
these latter two quantities are not altered by the state of $w$
because that state is fixed at $0$ by construction.  Hence, the phase
space of $X$ contains a cycle of length $\lcm(c_1, c_2)$ as claimed.
\end{proof}

Thus, for $\kup=1$ and $\kdown=3$, there exists a circle graph and
permutation $\pi$ that will produce a cycle in phase space of length
three or greater, and multiple circle graphs can be combined to
produce graphs with large orbit cycles without modifying the
thresholds of vertices in $X_1$ and $X_2$.


\subsection{Trees}

Propositions~\ref{prop:circle} and~\ref{prop:union} show how periodic
orbits of length $>2$ arise over graphs that contain cycles. This
section investigates bi-threshold SDS maps where~$X$ is a tree.

\medskip

To start, we first recall the notion of $\kappa$-equivalence of
permutations from~\cite{Macauley:09a,Macauley:08b}. Two permutations
$\pi,{\pi'}\in S_X$ are \emph{$\kappa$-equivalent} if the
corresponding \emph{induced acyclic orientations}~$O_\pi$
and~$O_{\pi'}$ of~$X$ are related by a sequence of source-to-sink
conversions. Here, the orientation $O_\pi$ is obtained from $\pi$ by
orienting each edge $\{v,v'\}\in\eset[X]$ as $(v,v')$ if~$v$ precedes
$v'$ in $\pi$ and as $(v',v)$ otherwise. This is an equivalence
relation, and it is shown in~\cite{Macauley:09a} that ($i$) for a tree
the number of $\kappa$-equivalence classes is $\kappa(X) = 1$, and
($ii$) that $\F_\pi$ and $\F_{\pi'}$ have the same periodic orbit
structure (up to digraph isomorphism/topological conjugation) whenever
$\pi$ and ${\pi'}$ are $\kappa$-equivalent. As a result, \emph{we only
  need to consider a single permutation update sequence} to study the
possible periodic orbit structures of permutation SDS maps over a
tree~$X$.

\medskip

The following result shows that there can be cycles of length~$3$
or greater for permutation SDS over a tree.

\begin{proposition}
\label{prop:h-tree}
For any integer $c \geq 3$ there is a tree $X$ on $n=4c-6$ vertices
such that bi-threshold permutation SDS maps over $X$ with thresholds
$\kup=1$ and $\kdown=3$ have periodic orbits of length $c$.
\end{proposition}

\begin{proof}
An $H$-tree on $n = 4\beta + 2$ vertices, denoted by $H_n$, has vertex
set $\{1,2,\ldots,n\}$ and edge set
\begin{equation*}
\{\eta,n-\eta+1\} \cup
\bigl\{ \{i,i+1\}, n/2 + \{i,i+1\} \mid 1\le i \le n/2-1  \bigr\} \;,
\end{equation*}
where $\eta=\beta+1$ and $\beta \ge 1$.  The graph $H_{n}$ is
illustrated in Figure~\ref{fig:h-tree}.
\begin{figure}
  \centerline{\includegraphics[]{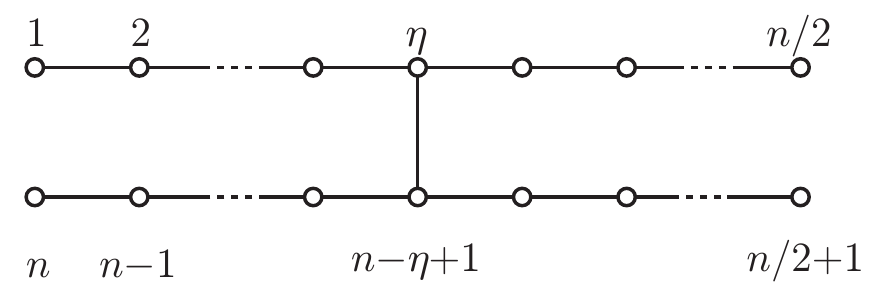}}
  \label{fig:h-tree}
  \caption{The tree $H_n$ used in the proof of Proposition~\ref{prop:h-tree}.}
\end{figure}

\medskip

Set $\beta=c-2$ so that $n=4\beta+2$ and $\eta=\beta+1$.  We take $X =
H_n$ as the graph and assign thresholds $(\kup,\kdown)=(1,3)$ to all
vertices.
By the comment preceding Proposition~\ref{prop:union}, we may simply
use $\pi=(1, 2, 3, \ldots, n)$ as update sequence since all
permutations give cycle equivalent maps $\F_\pi$.

For the initial configuration, set the state of each vertex $v$ in the range
$(n/2)+1 \le v \le n-\eta+1$ (bottom right branch) to~$1$ and
set all other vertex states to~$0$ so that
\begin{equation*}
x(0) = (0, 0, \dots, 0,
     \underbrace{1 , 1, \dots, 1,}_{{\text{start at vertex $(n/2)+1$}}}
   0, 0, \dots 0 )
\end{equation*}

The number of vertices in a contiguous vertex range with state 1 will
always be $\eta$; there may be one or two such groups in a system
state. The image of $x(0)$ is
\begin{equation*}
x(1)=(0, 0, \dots, 0,
     \underbrace{1 , 1, \dots, 1,}_{\text{start at vertex $\eta$}}
   0, 0, \dots 0) \;,
\end{equation*}
where now the first $\eta-1$ vertices are in state 0, the next $\eta$
vertices are in state 1, and the remaining vertices---all those along
the bottom arm---are in state 0, as follows.  Along the top arm,
vertices~$1$ through~$\eta-1$ will remain in state~$0$ because all
nodes and their neighbors are in state~$0$.  Vertex~$\eta$, the state
of the vertex incident to the crossbar on the top arm, will change
to~$1$ because its neighbor along the crossbar is in state~$1$.  For
the given permutation, then, each subsequent vertex $v_i$ in the range
$\eta+1$ through $n/2$ will change to state~$1$ because $x_{v_{i-1}}=1$
and $\kup=1$.  For the bottom arm, vertex $(n/2)+1$ will change from
state 1 to state 0 because $\sigma(x[v_{(n/2)+1}])=2 < \kdown$.  For
the same reason, each vertex $v_i$ in the range $(n/2)+2$ to
$n-\eta+1$ will transition to state 0.  Vertices from $n-\eta+2$
through $n$ will remain in state 0.

The next state is
\begin{equation*}
x(2)=(0, 0, \dots, 0,
     \underbrace{1 , 1, \dots, 1}_{\text{start at vertex }\eta-1},
     0, 0, \ldots 0,
     \underbrace{1 , 1, \dots, 1}_{\text{start at vertex } n-\eta+1}) \;,
\end{equation*}
where, for the top arm, the first $\eta-2$ vertices are in state 0,
the next $\eta$ vertices are in state 1, and the last vertex on the
top arm is in state 0.  That is, the set of 1's along the top arm has
shifted one vertex left, as follows.  Let the set of vertices in the
top arm in state 1 (in $x(1)$) be denoted $v_i$ through $v_{i+\eta}$.  Vertex
$v_{i-1}$ will transition $0 \rightarrow 1$ because $x_{v_i}=1$.
Vertex $v_i$ will remain in state 1 because $\sigma(x[v_i])=3=\kdown$.
Likewise $v_{i+1}$ through $v_{i+\eta-1}$ will remain in state 1.
However, $v_{i + \eta}$ will transition to state 0 because
$\sigma(x[v_{i + \eta}])=2<\kdown$.  We refer to this behavior as a
\emph{left-shift} (the analogous shift to the right is a
\emph{right-shift}).  For the bottom arm, the $\eta$ vertices (labels
$(n/2)+1$ through $n-\eta$) remain in state 0.  Vertex $n-\eta+1$
transitions to state 1 because the neighbor along the
crossbar is in state 1.  Subsequently, vertices $n-\eta+2$
through $n$ transition to state 1, in turn, according to $\pi$.

The next state is
\begin{equation*}
x(3)=(0, 0, \ldots, 0,
      \underbrace{1 , 1, \dots, 1}_{\text{start at vertex }\eta-2},
      0, 0, \ldots 0,
      \underbrace{1 , 1, \dots, 1}_{\text{start at vertex } n-\eta}) \;,
\end{equation*}
where the set of $\eta$ vertices in state 1 in the top arm has
shifted left, and the set of $\eta$ vertices in state 1 in the bottom arm has
shifted left.  The shifting process embodied in the transition from
state $x(2)$ to $x(3)$---where there is a group of vertices in state 1
in each of the top and bottom arms---can happen a total of $(\eta-2)$
times.  The state after these $(\eta-2)$ transitions is
\begin{equation*}
x(\eta)=(\underbrace{1 , 1, \dots, 1}_{\text{start at vertex 1}},
         0, 0, \ldots 0,
         \underbrace{1 , 1, \dots, 1,}_{\text{start at vertex } n-2\eta+3}
         0,0, \ldots, 0)\;.
\end{equation*}
The image of $x(\eta)$ is $x(0)$, the initial state.  There are
$2+(\eta-2)+1$ state transitions, and we have a limit cycle of length
$c=\eta+1$.
\end{proof}

Of course, the proof does not guarantee that $c$ is the minimal
periodic orbit size, nor that $H_n$ is the minimal order tree with a
periodic orbit of this length. Additionally, there may be multiple
periodic orbits of length $c$. The following proposition expands on
this in the case where $c\geq5$: there exists a tree of smaller order
than $H_n$ that also admits a $c$-cycle, namely the $Y$-trees.

\begin{proposition}
\label{prop:y-tree}
For any integer $c \geq 3$ there is a tree on $n=3c-2$ vertices such
that bi-threshold permutation SDS maps over this tree with thresholds
$\kup=1$ and $\kdown=3$ have periodic orbits of length $c$.
\end{proposition}

\begin{proof}
The proof is analogous to the case of the $H$-tree. We take as the graph
the $Y$-tree on $n = 3\beta + 1$ vertices (see
Figure~\ref{fig:y-tree}) with $\beta\ge1$, which has vertex set
$\{1,2,\ldots,n\}$ and, setting $\eta = \beta+1$, edge set
\begin{equation*}
\bigl\{ \{i,i+1\} \mid 1\le i \le 2\eta - 2  \bigr\}
\cup
\bigl\{ \{i,i+1\} \mid 2\eta \le i \le (n-1)  \bigr\}
\cup
\{\eta, n\} \;.
\end{equation*}
\begin{figure}
  \centerline{\includegraphics[]{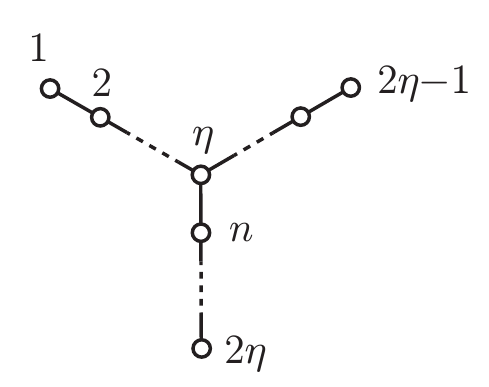}}
  \caption{The tree $Y_n$ used in the proof of
    Proposition~\ref{prop:y-tree}.}
    \label{fig:y-tree}
\end{figure}
Let $c\geq3$ with $n=3c-2$ so that $X = Y_n$ (and $c = \beta+1$).
We assign thresholds $(\kup,\kdown)=(1,3)$ to all vertices and
use update sequence $\pi=(1, 2, 3, \ldots, n)$ as before.  As the
initial configuration, set the states of the $\beta$ vertices $v$ in
the range $\eta \le v \le 2\eta-2$ (all vertices in the upper right
branch except $2\eta-1$) to~$1$, and set all other vertex states
to~$0$ to form
\begin{equation*}
x(0) = (0, 0, \ldots, 0,
     \underbrace{1 , 1, \dots, 1,}_{\text{start at vertex $\eta$}}
     0, 0, \dots 0 ) \;.
\end{equation*}
The image of $x(0)$ is
\begin{equation*}
x(1)=(0, 0, \dots, 0,
     \underbrace{1 , 1, \dots, 1,}_{\text{start at vertex $(\eta-1)$}}
   0, 0, \dots 0,1) \;,
\end{equation*}
where now the first $\eta-2$ vertices are in state~$0$, the next
$\beta$ vertices are in state 1, and the remaining vertices---except
for vertex $n$---are in state~$0$.  In the upper two branches, the
initial set of $\beta$ nodes in state~$1$ \emph{shifts left} for the
same reasons described in the proof of Proposition~\ref{prop:h-tree}.
The last vertex, $n$, will change to~$1$ because it is adjacent to
vertex $\eta$, which has state~$1$.

The image of $x(1)$ is
\begin{equation*}
x(2)=(0, 0, \ldots, 0,
     \underbrace{1 , 1, \dots, 1,}_{\text{start at vertex $(\eta-2)$}}
   0, 0, \ldots 0,1,1) \;,
\end{equation*}
where the $\beta$ nodes in state~$1$ beginning at vertex $\eta-2$ have
shifted left and vertex $n-1$ transitions to~$1$ because vertex $n$ is
in state~$1$.  Vertex~$n$ remains in state~1 because $\sigma(x[v_n])=3$.

The mechanics of the last state transition (the left shift of $\beta$
vertices and nodes transitioning to state~$1$ in the lower branch)
repeats itself a total of $\beta-2$ times, at which point the state is
\begin{equation*}
x(\beta-1)=(0, 1, \dots, 1,
     \underbrace{0 , 0, \dots, 0,}_{\text{start at vertex $(\eta+1)$}}
   \underbrace{1, 1, \dots, 1}_{\text{start at vertex $(n-\beta+2)$}})\;,
\end{equation*}
where the only vertex in the lower vertical branch in state~$0$ is
$2\eta$, the leaf node.

Noting that vertex $\eta$ remains in state~$1$ on the next transition
because $\sigma(x[v_\eta])=3$, all vertices in the upper right branch
transition to~$1$.  Vertex $2\eta$ also transitions to~$1$,
giving
\begin{equation*}
x(\beta)=(1, 1, \ldots, 1) \;.
\end{equation*}
The next state can be verified to be $x(0)$, thus completing the
cycle.  The cycle length is therefore $c = \beta+1$ as stated.
\end{proof}

Interestingly, there is no $H$-tree nor $Y$-tree that generates a
maximum orbit of size~2 for thresholds $(\kup,\kdown)=(1,3)$. However,
so-called $X$-trees (defined below) admit cycles of any size $c \ge 1$.

\begin{proposition}
\label{prop:x-tree}
For any integer $c \geq 2$ there is a tree $X$ on $n=4c-3$ vertices
such that bi-threshold permutation GDS maps over $X$ with thresholds
$\kup=1$ and $\kdown=3$ have periodic orbits of length $c$.  For
$c=1$, there is a tree $X$ on $n=5$ vertices that has periodic orbits
of length~1 (fixed points).
\end{proposition}

\begin{proof}
An $X$-tree on $n = 4\beta + 1$ vertices with $\beta\ge1$ has vertex
set $\{1,2,\ldots,n\}$ and edge set as illustrated in
Figure~\ref{fig:x-tree}. Here $\eta=\beta+1$ is the unique
vertex of degree~4.
\begin{figure}
  \centerline{\includegraphics[]{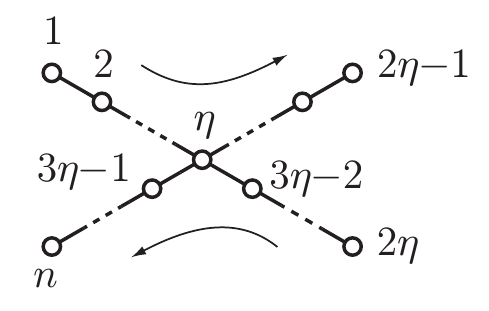}}
  \caption{The tree $X_n$ used in the proof of
    Proposition~\ref{prop:x-tree} (arrows indicate vertex labeling
    order).}
  \label{fig:x-tree}
\end{figure}
Note first that for any $n$ the all-zero state over $X_n$ is a fixed
point.

\medskip

We treat the case $c=2$ separately; use $X=X_5$, $\pi=(1,2,3,4,5)$, and
$(\kup,\kdown)=(1,3)$. It can easily be verified that
$x(0)=(0,1,1,0,0)$ is mapped to $x(1)=(1,1,0,1,1)$ which in turn is
mapped to $x(0)$, constituting a $2$-cycle.

\medskip

Fix $c \geq 3$, set $n=4c-3$ and then $c=\beta+1$, take as the graph $X = X_n$ with thresholds
$(\kup,\kdown)=(1,3)$ for all vertices, and let $\pi=(1, 2, 3, \ldots,
n)$.

Define the initial configuration $x(0)$ by assigning the $\beta$
vertices $v$ with $\eta \le v \le 2\eta-2$ (all vertices in the upper
right branch except $2\eta-1$) to~$1$ and set all other vertex states
to~$0$, that is,
\begin{equation*}
x(0) = (0, 0, \dots, 0,
     \underbrace{1 , 1, \dots, 1,}_{\text{start at vertex $\eta$}}
   0, 0, \dots 0 ) \;.
\end{equation*}
The image of $x(0)$ is
\begin{equation*}
x(1)=(0, 0, \dots, 0,
     \underbrace{1 , 1, \dots, 1,}_{\text{start at vertex $(\eta-1)$}}
   0, 0, \dots 0,\underbrace{1 , 1, \dots, 1}_{\text{start at vertex $3\eta-2$}}) \;,
\end{equation*}
where now the first $\eta-2$ vertices are in state 0, the next $\beta$
vertices are in state 1, and the remaining vertices in branch~2 are in
state~$0$.  In branch~3, only the vertex neighboring vertex $\eta$
transitions to state~$1$, while all vertices in branch~4 transition to
state~$1$ because $\eta$ is in state~$1$.

State $x(2)$ is generated by a left-shift of the $\beta$ contiguous
states that are 1 in branches 1 and 2, and by a left-shift of the
$\beta+1$ contiguous state-1 vertices in branches~3 and~4, that is,
\begin{equation*}
x(2)=(0, 0, \ldots, 0,
     \underbrace{1 , 1, \dots, 1,}_{\text{start at vertex $(\eta-2)$}}
     0, 0, \ldots 0,
     \underbrace{1 , 1, \dots, 1,}_{\text{start at vertex $3\eta-3$}}0) \;.
\end{equation*}
From $x(1)$ there are $\beta-2$ such transitions that result in the
state
\begin{equation*}
x(\beta-1)=(0,
     \underbrace{1 , 1, \dots, 1,}_{\text{start at vertex $2$}}
   0, 0, \dots 0,\underbrace{1 , 1, \dots, 1,}_{\text{start at vertex $3\eta-\beta$}}
   0, 0, \dots, 0) \;.
\end{equation*}
The next transition results in all vertices in branches 1 and 2 in
state~$1$ since $\eta$ remains in state 1.  The contiguous
set of $\beta+1$ vertices in branches 3 and 4 shift left, giving
\begin{equation*}
x(\beta)=(1 , 1, \dots, 1,
   \underbrace{0 , 0, \dots, 0}_{\text{start at vertex $3\eta$}}) \;.
\end{equation*}
The image of $x(\beta)$ is $x(0)$, and, since $\beta$ is the smallest
positive time step with this property, we have established the
presence of a periodic orbit of length $c=\beta+1$.
\end{proof}

Finally, we consider a special class of bi-threshold SDSs on trees
with $\kup=1$ and $\kdown = \kdown(v) = d(v)+1$ for each vertex $v$.
Note that the down-threshold for each vertex depends on its degree as
indicated by the index $v$ in $\kdown(v)$. We show that such
bi-threshold SDS maps always have fixed points. In such systems, the
state of a vertex $v$ switches from 0 to 1 if it has at least
one neighbor in state 1, and from 1 to 0 if it has at least one
neighbor in state 0. This is an interesting contrast to the classes of
bi-threshold SDSs on trees discussed above which have large limit
cycles.

Let $X$ be a tree. We choose some arbitrary vertex $r\in v[X]$ as its
root, and partition $X$ into levels $X_0, X_1,\ldots, X_D$ with
respect to $r$ such that $X_0=\{r\}$, and for any $i\geq 0$, we let
$X_{i+1}$ be the set of vertices adjacent to vertices in set $X_i$,
but not in the set $\cup_{j<i} X_j$. We sometimes refer to $X_i$ as
level-$i$ set. Let $D$ be the number of levels. We can also define a
parent-child relationship relative to this rooted tree, and denote
$p(v)$ as the parent of vertex $v\neq r$.  In our arguments below, we
use any permutation $\pi$ of $v[X]$, which consists of all the
vertices in $X_i$ before those in $X_{i-1}$ for each $i$.  Our result
is based on the following property.

\begin{lemma}
\label{lemma:pv}
Consider a bi-threshold SDS $\F_\pi$ on a tree $X$ with an arbitrary
root $r$ and permutation $\pi$ as defined above where $\kup=1$ and
$\kdown(v) = d(v)+1$ for each vertex $v$.  Let $x$ be any state vector
and $x'=\F_\pi(x)$. For each vertex $v$ other than the root, we have
$x'_v = x_{p(v)}$.
\end{lemma}
\begin{proof}
Our proof is by induction on the levels, starting from the highest,
i.e., $X_D$. For the base case, consider a leaf $v\in X_D$. We have
four cases: $x_v=x_{p(v)}=1$, $x_v=0, x_{p(v)}=1$, $x_v=1, x_{p(v)}=0$
and $x_v=x_{p(v)}=0$.  It is easy to verify that in the first two
cases, we have $x'_v=1$ and in the latter two cases, we have $x'_v=0$,
since vertex $v$ is updated before $p(v)$ in $\pi$. Therefore, the
statement of the lemma holds in the base case for all vertices $v\in
X_D$.

Next, consider a vertex $v$ in some level $X_j$, $j<D$. If $v$ is a
leaf in $X_j$, the lemma follows by exactly the same argument as in
the base case. Therefore, consider the case $v$ is not a leaf. Let
$w_1,\ldots,w_c$ denote its children.  Since level $j+1$ vertices are
updated before those in level $j$ in $\pi$, by induction, we have
$x'_{w_i}=x_v$ for each $w_i$. Again, we have a case similar to the
base case: when vertex $v$ is updated, it has the same values as its
children, and therefore, takes on the state of $p(v)$. Thus, the lemma
follows.
\end{proof}

This property immediately gives us the following:
\begin{corollary}
Let $X$ be a tree. Let $\pi\in S_x$ and let $(f_v)_v$ be bi-threshold
functions satisfying $\kup=1$ and $\kdown(v)=d(v)+1$ for each
vertex $v$. Any SDS map $\F_\pi$ only has
fixed points as limit sets.
\end{corollary}
\begin{proof}
Without loss of generality, we take $\pi$ to be the permutation in
Lemma~\ref{lemma:pv}.  By applying Lemma~\ref{lemma:pv}, it is easy to
verify that for any state vector $x$, all the vertices in levels 0 and
1 have the same state value in $F(x)$, namely $x_r$. By induction on
$i$, it is easy to verify that for any $i\geq 1$, all vertices in
levels $0,\ldots,i$ have the same state value (of $x_r$) in
$F^i(x)$. The statement follows since all permutations for a tree give
cycle equivalent SDS maps.
\end{proof}


\section{Summary and Conclusion}
\label{sec:conclude}

This paper has analyzed the structure of $\omega$-limit sets of
bi-threshold GDS. Unlike the synchronous case, bi-threshold SDS maps
can have long periodic orbits, and this is characterized in terms of
the difference of the up- and down-thresholds. We also analyzed
certain classes of trees. The following is a list of questions and
conjectures for possible further research.

\subsection{Embedding and Inheritance of Dynamics}

A fundamental question in the study of GDSs is the following: if a
graph $X$ has a graph $X'$ as an induced subgraph, what are the
relations between the dynamics over the two graphs? Here one has to
assume that the vertex function, and update sequences if applicable,
are appropriately related. For example, is there a projection from the
phase space of the GDS over $X$ to the one over $X'$?

In initial computational experiments we studied the dynamics for
bi-threshold GDS over trees obtained from, e.g. $H$-trees by adding a
collection of edges - results indicate that there are several classes
of outcomes. While this is hardly a surprise, there are clear patterns
in how edges are added and the dynamics that result. For example, some
classes of edge additions give trees that have long periodic orbits
just as in the case of $H$-trees. For other classes of edge additions,
however, the addition of even a single edge causes all periodic orbits
of size $\ge 2$ to disappear.
Further insight into the mechanisms involved could shed light on the
the fundamental question above.

\subsection{Minimality of Trees with Given Periodic Orbit Sizes}

Our results above on the existence of trees admitting bi-threshold SDS
with given periodic orbit sizes are not necessarily minimal. For a
given $c\ge 1$ there is an $X$-tree with a periodic orbit of length
$c$, but there may be a smaller tree (or graph in general) which
admits periodic orbits of size~$c$ as well. While we have obtained
some insight on this via sampling, no firm results have been
established.

\medskip

\noindent\textbf{Note.} For all computational experiments involving
dynamics of SDS maps over graphs in this paper we used a variant of
InterSim~\citep{cjk:kkmmstrr-2011d}.


\acknowledgements
\label{sec:ack}

We thank our external collaborators and members of the Network
Dynamics and Simulation Science Laboratory (NDSSL) for their
suggestions and comments.  This work has been partially supported by
NSF Nets Grant CNS-0626964, NSF HSD Grant SES-0729441, NSF PetaApps
Grant OCI-0904844, NSF NETS Grant CNS-0831633, NSF REU Supplement
Grant CNS-0845700, NSF Netse Grant CNS-1011769, NSF SDCI Grant
OCI-1032677, DTRA R\&D Grant HDTRA1-0901-0017, DTRA CNIMS Grant
HDTRA1-07-C-0113, DOE Grant DE-SC0003957, US Naval Surface Warfare
Center Grant N00178-09-D-3017 DEL ORDER 13, NIH MIDAS project
2U01GM070694-7 and NIAID \& NIH project HHSN272201000056C.


\appendix

\section{Limit Cycle Structure for Standard Threshold Cellular Automata}
\label{sec:standard}

This appendix section contains a condensed version of the proof
from~\cite{Goles:81} for standard threshold functions. We have
incorporated their proof for two reasons. First, only a portion of the
original proof needs to be adapted to cover bi-threshold systems, and
in this way the paper becomes self-contained. Second, the original
proof only appears in French, and we here provide an English version.
Let $K = \{0,1\}$, let $A=(a_{ij})_{i,j=1}^n$ be a real symmetric matrix, let
$\theta = (\theta_1,\dots,\theta_n) \in\R^n$, and let $\F = (f_1,
\ldots, f_n) \colon K^n \longrightarrow K^n$ be the function defined
coordinate-wise by
\begin{equation}
  \label{eq:genthreshold}
  f_i(x_1,\ldots,x_n) =
  \begin{cases}
    0,  & \text{if } \sum\limits^n_{j=1} a_{ij} x_j < \theta_i  \\
    1,  & \text{otherwise} \;.
  \end{cases}
\end{equation}

\begin{theorem}
\label{thm:thres-gca}
For all $x \in K^n$, there exists $s\in\N$ such that $\F^{s+2}(x) =
\F^{s}(x)$.
\end{theorem}

The proof of this theorem is based on two lemmas which are given
below. Note first that since $K^n$ is finite, for each $x \in K^n$
there exist $s,T\in\mathbb{N}$ (they will generally depend on $x$) with
$T>0$ such that
\begin{equation*}
  \F^{s+T}(x) = \F^{s}(x)
  \quad \text{and} \quad
  \F^{s+r}(x) \neq \F^{s}(x)
\end{equation*}
for all $0 < r < T$. Here~$s$ is the transient length of the state~$x$.
Next define the $n \times T$ matrix $X(x,T) = (\F^{s}(x),
l\dots, \F^{s+T-1}(x))$ by
\begin{displaymath}
X(x,T)=\left(
\begin{array}{ccc}
z_1(0)  &\dots & z_1(T-1)  \\
\vdots &\cdots &\vdots\\
z_n(0)  &\dots & z_n(T-1)
\end{array}
\right)\;,
\end{displaymath}
where $\F^{s}(x) = z = (z_1(0),\dots,z_n(0))$ and $\F^{s+T-1}(x) =
(z_1(T-1),\dots,z_n(T-1))$. In other words,~$z$ denotes the first
periodic point reached from $x$ (after~$s$ steps) and its period
is~$T$. The columns of $X(x,T)$ are the $T$ successive periodic points
of the cycle containing $z$.


In general we have
\begin{displaymath}
\F^{s+l}(x) = (z_1(l),\dots,z_n(l)) \text{ for } 0 \le l \le T-1 \;.
\end{displaymath}
Since
\begin{displaymath}
\F^{s}(x) =  \F^{s+T}(x) = \F(z_1(T-1), \ldots, z_n(T-1))
\end{displaymath}
we have $z_i(0) = f_i(z_1(T-1),\ldots,z_n(T-1))$, and from
$\F^{s+l+1}(x) = \F(\F^{s+l}(x))$ we have
\begin{equation*}
  z_i(l+1) = f_i(z_1(l),\dots,z_n(l)) \text{ for } l=0,\ldots,T-2.
\end{equation*}
We will call $z_i$ the $i{}^{\text{th}}$ row of the matrix $X(x,T)$
and let $\gamma_i$ denote the smallest divisor of $T$ such that
$z_i(l+\gamma_i) = z_i(l)$ for $l\in\{0,\dots,T-1\}$, and will say
that $\gamma_i$ is the period of the component $z_i$.  Clearly, we
have $z_i(l+T) = z_i(l)$ for $i\in\{1,2,\ldots,n\}$ and all
$l\in\{0,\dots,T-1\}$.
Let $S = \{z_1,\dots,z_n\}$ be the set of rows of~$X(x,T)$. We define
the operator $L \colon S\times S\rightarrow\mathbb{R}$ by
\begin{equation*}
  L(z_i, z_j) = a_{ij} \sum^{T-1}_{l=0} (z_j(l+1) - z_j(l-1)) z_i(l)\;,
\end{equation*}
with indices taken modulo~$T$.
\begin{lemma}
\label{lem:tprop1}
The operator $L$ has the following properties:
\begin{itemize}
\item[($i$)] $L(z_i,z_j)+L(z_j,z_i) = 0$ for $i,j \in \{1,\ldots,n\}$
  (anti-symmetry).
\item[($ii$)] If $\gamma_i\le 2$ then $L(z_i,z_j)=0$ for $j \in
  \{1,\ldots,n\}$.
\end{itemize}
\end{lemma}
\begin{proof}
For~($i$), since $a_{ij} = a_{ji}$, we have
\begin{align*}
L(z_i, z_j) + L(z_j, z_i) =
  a_{ij} \sum^{T-1}_{l=0} \bigl(
     &[z_i(l)z_j(l+1) - z_i(l-1)z_j(l)] \\+& [z_i(l+1)z_j(l) - z_i(l)z_j(l-1)]
   \bigr)\;,
\end{align*}
which clearly evaluates to zero due to periodicity. For part~($ii$), if
$\gamma_i = 1$ then the row $z_i$ is constant and $L(z_i, z_j) = 0$. If
$\gamma_i = 2$ then the value of $z_i$ alternates as
\begin{equation*}
z_i(0), z_i(1), z_i(0), z_i(1), \ldots, z_i(0), z_i(1)
\end{equation*}
across the $i^{\text{th}}$ row, and the terms in $L(z_i, z_j)$ cancel
in pairs.
\end{proof}

Let $z_i\in S$ and suppose in the following that $\gamma_i \ge 3$. We
set
\begin{equation*}
  \supp(z_i) = \{ l\in\{0,\ldots,T-1\} : z_i(l) = 1\}\;,
\end{equation*}
and write $\mathcal{I}(l) = \{l, l+2, l+4, \ldots, l-4, l-2\}$. Next, set
\begin{equation*}
C_0 =
\begin{cases}
\varnothing,& \text{ if there is no }l_0\in\{0,\dots,T-1\} \text{ such
  that } \mathcal{I}(l_0) \subset \supp(z_i)\\
  \mathcal{I}(l_0),& \text{otherwise}.
\end{cases}
\end{equation*}
We define $C_1$ as the set
\begin{equation*}
C_1 = \{l_1+2s \in \supp(z_i) : s = 0,1,\ldots, q_1\} \;,
\end{equation*}
where $l_1$ is the smallest index not in $C_0$ satisfying $z_i(l_1-2)
= 0$ and $q_1$ satisfies $z_i(l_1+2q_1+2) = 0$.
For $k \ge 2$ we define the sets $C_k$ by
\begin{equation*}
C_k = \{ l_k + 2s \in \supp(z_i) : s = 0,1,\dots,q_k\}\;,
\end{equation*}
where $l_k = l_{k-1} + r \pmod T \notin \{l_1,\ldots,l_{k-1}\}$ is
the smallest index for which $z_i(l_k-2) = 0$ and $q_k$ satisfies
$z_i(l_k+2q_k+2) = 0$.

Since $\gamma_i \ge 3$ (assumption), there always exists $l_1\in
\supp(z_i)$ for which $z_i(l_1-2) = 0$. This allows us to build the
collection of sets $\mathcal{C} = \{C_0,\dots,C_p\}$. By construction,
$\mathcal{C}$ is a partition of $\supp(z_i)$.
The following lemma provides the final piece needed in the proof of
the main result.
\begin{lemma}
\label{lem:tprop2}
For $z_i\in S$ and with $\gamma_i \ge 3$ we have
\begin{equation*}
  \sum^n_{j=1}L(z_i,z_j)<0 \;.
\end{equation*}
\end{lemma}
\begin{proof}
Using the partition $\mathcal{C}$ of $\supp(z_i)$, we have
\begin{align*}
\sum^n_{j=1} L(z_i, z_j)
     &= \sum^n_{j=1} a_{ij} \sum_{l\in \supp(z_i)} (z_j(l+1)-z_j(l-1)) \cdot 1 \\
     &= \sum^n_{j=1} a_{ij} \sum^p_{k=0} \sum_{l\in C_k} (z_j(l+1)-z_j(l-1))
     = \sum^p_{k=0} \sum^n_{j=1} a_{ij} \sum_{l\in C_k} (z_j(l+1)-z_j(l-1)) \\
     &= \sum^p_{k=0} \Psi_{ik}\;,
\end{align*}
where we have introduced
\begin{equation}
\label{eq:Psi}
\Psi_{ik} = \sum^n_{j=1} a_{ij} \sum_{l\in C_k} (z_j(l+1)-z_j(l-1)) \;.
\end{equation}
If $C_0 = \emptyset$ then $\Psi_{i0} = 0$, and if $C_0 = \{l_0,l_0+2,
\ldots, l_0-2\}$ we have
\begin{equation*}
\sum_{l\in C_0} (z_j(l+1)-z_j(l-1)) = 0\;.
\end{equation*}
In other words, we always have $\Psi_{i0} = 0$, so we assume $k>0$ in
the following.
From the assumption that $\gamma_i \ge 3$, there exists $C_k\ne
\varnothing$ such that $C_k = \{l_k,l_k+2,\ldots,l_k+2q_k\}$, so we
can re-write $\Psi_{ik}$ as
\begin{align*}
\Psi_{ik} &= \sum^n_{j=1}a_{ij}\sum_{s=0}^{q_k} (z_j(l_k+2s+1)-z_j(l_k+2s-1))\\
         &= \sum^n_{j=1}a_{ij} z_j(l_k+2q_k+1) - \sum^n_{j=1} a_{ij} z_j(l_k-1)\;.
\end{align*}
{\bfseries [Cross-reference for bi-threshold systems]} By the construction of
$C_k$, we have $z_i(l_k+2q_k+2) = 0$ and $z_i(l_k) = 1$ which, by the
definition of $f$ in~\eqref{eq:genthreshold}, is only possible if
\begin{equation}
\label{eq:synch_cond}
\sum^n_{j=1} a_{ij} z_j(l_k+2q_k+1)<\theta_i, \quad\text{and}\quad
\sum^n_{j=1}a_{ij}  z_j(l_k-1)\ge \theta_i\;.
\end{equation}
This implies that $\Psi_{ik}<0$ and we conclude that
\begin{equation*}
\sum^n_{j=1} L(z_i,z_j) = \sum^p_{k=1}\Psi_{ik} < 0 \;
\end{equation*}
as required.
\end{proof}

\begin{proof}[of Theorem~\ref{thm:thres-gca}]
From Lemma~\ref{lem:tprop1} we have that~$L$ is anti-symmetric so
\begin{equation*}
\sum^n_{i=1}\sum^n_{j=1}L(z_i,z_j) = 0 \;.
\end{equation*}
However, if we assume that $T\ge 3$, then there is $z_i$ with
$\gamma_i \ge 3$ and Corollary~\ref{cor:m01-m10} produces the desired
contradiction. We conclude that $T\le 2$.
\end{proof}








\bibliographystyle{abbrvnat}
\bibliography{refs,contagion-refs}
\label{sec:biblio}

\end{document}